\pdfoutput=1
\documentclass[a4paper,11pt]{article}
\usepackage{amsthm}

\usepackage[
  textheight=236mm,  
  textwidth=156mm,
]{geometry}

\usepackage{amsmath, amssymb, amsfonts}
\usepackage{bbm}
\usepackage{latexsym}
\usepackage{color}
\usepackage{soul}
\usepackage{cite}
\usepackage{xcolor}
\usepackage{microtype}
\usepackage{cleveref}

\newcommand{\xc}{\mathop\mathrm{xc}\nolimits}
\newcommand{\rkn}{\mathop\mathrm{rk_+}\nolimits}
\newcommand{\hsb}{\mathop\mathrm{hsb}\nolimits}
\newcommand{\conv}{\mathop\mathrm{conv}\nolimits}
\newcommand{\aff}{\mathop\mathrm{aff}\nolimits}
\newcommand{\R}{\mathbb{R}}
\newcommand{\N}{\mathbb{N}}

\renewcommand{\subset}{\subseteq}

\newcommand{\norm}[2][\infty]{\|#2\|_#1}
\newcommand{\Pct}{P_\mathrm{ct}}
\newcommand{\Pst}{P_\mathrm{st}}
\newcommand{\Symm}{\mathfrak{S}_n}

\newtheorem{theorem}{Theorem}
\newtheorem{lemma}{Lemma}

\newtheorem{proposition}{Proposition}
\crefname{theorem}{Theorem}{Theorems}
\crefname{lemma}{Lemma}{Lemmas}
\crefname{corollary}{Corollary}{Corollaries}
\crefname{proposition}{Proposition}{Propositions}
\crefname{equation}{}{}
\crefrangelabelformat{equation}{\textup{(#3#1#4)--(#5#2#6)}}   % 
\crefname{section}{Section}{Sections}
\usepackage{enumitem}
\newlist{myenumerate}{enumerate}{1}
\setlist[myenumerate]{
  label={\itshape(\roman*)},  % label format (i),(ii),...
  ref={\itshape(\roman*)},    % same reference format
}
\crefname{myenumeratei}{}{}

\begin{document}

\title{
  Limitations of the Hyperplane Separation Technique  for Bounding the Extension Complexity of Polytopes%  
}

\author{Matthias Brugger%
  \thanks{Operations Research, Department of Mathematics, Technische Universit\"at M\"unchen, Germany. \newline 
  E-mail: \texttt{matthias.brugger@tum.de}.
  \newline
  Supported by the Alexander von Humboldt Foundation with funds from the German Federal Ministry of Education and Research (BMBF).}
}

\date{}

\maketitle

\begin{abstract}
We illustrate the limitations of the hyperplane separation bound,
a non-combinatorial lower bound on the extension complexity of a polytope.
Most notably, this bounding technique is used by Rothvo{\ss} (J ACM 64.6:41, 2017)
to establish an exponential lower bound for the perfect matching polytope.
We point out that the technique is sensitive to the particular choice of slack matrix.
For the canonical slack matrices of the spanning tree polytope and the completion time polytope,
we show that the lower bounds produced by the hyperplane separation method are trivial.
These bounds may, however, be strengthened by normalizing rows and columns of the slack matrices.
\end{abstract}

%-------------------------------------------------------------------------------------

\section{Introduction} \label{sec:intro}

The \emph{extension complexity} of a polytope $P$, denoted by $\xc(P)$, is the minimum number of facets 
of any polytope $Q$ that affinely projects onto $P$.
A linear description of such a polytope $Q$ (together with the corresponding 
projection) is an \emph{extended formulation} of $P$.
If we define the \emph{size} of an extended formulation as the number of its inequalities,
the minimum size of any extended formulation of $P$ equals $\xc(P)$.

% -----

Building on Yannakakis' seminal work \cite{Yan}, there has recently been a renewed interest in the study of extended 
formulations (see, e.g., \cite{AF,FFG+,FKP+,FMP+,FRT,KPT,KW,Rot1,Rot}). For many polytopes associated with \textbf{NP}-hard combinatorial optimization problems, we now know that their extension complexity
cannot be bounded by a polynomial in their dimension; among them are TSP polytopes, cut and correlation polytopes, and stable set polytopes \cite{FMP+,KW}.
An exponential lower bound also holds for the extension complexity of the (perfect) matching polytope \cite{Rot}
(even though one can optimize over it in polynomial time).
Well-known polytopes that do admit nontrivial polynomial-size extended formulations include, among many others, parity polytopes \cite{Yan,CK}, independence polytopes of regular matroids \cite{AF}, and two families of polytopes considered here, spanning tree polytopes and completion time polytopes.
We refer to the surveys by Conforti et al.\ \cite{CCZ} and Kaibel \cite{Kai} for an overview and more examples.

% -----

The \emph{spanning tree polytope} of a connected graph $G=(V,E)$ is the convex hull of the incidence vectors of the spanning trees in $G$,
\begin{equation} \label{eq:st_polytope}
  \Pst(G) :=
  \conv \left\{
    \chi(T) \in \{0,1\}^E  \colon T \subset E \text{ is a spanning tree in } G 
  \right\},
\end{equation}
where $\chi(T)$ denotes the incidence vector of $T$.
Although $\Pst(G)$ has exponentially many facets in general,
there are extended formulations of size $O(|V| \, |E|)$ due to Wong \cite{Won} and Martin \cite{Mar} (see also \cite{Yan,CKW+}).
Special classes of graphs admit even smaller extended formulations:
For instance, Williams \cite{Wil} gives a formulation of size $O(|V|)$ for planar graphs.
Some progress has also been made for graphs of bounded genus, more generally, by Fiorini et al.\ \cite{FHJ+}.

On the other hand, it is known that the extension complexity of a polytope is at least its dimension \cite{FKP+}.
Thus, if $K_n$ is the complete graph on $n$ vertices, $\Omega(n^2)$ is a trivial lower bound on the extension complexity of $\Pst(K_n)$.
The question whether this bound can be improved is open \cite{Wel}.
Khoshkhah and Theis \cite{KT} show that the so-called \emph{rectangle covering lower bound} achieves at most $O(n^2 \log n)$.
This is a \emph{combinatorial} lower bound, that is, one that depends only on the vertex-facet incidence structure of the 
polytope and, thus, is unable to distinguish between combinatorially equivalent polytopes \cite{FKP+}.
In \cite{KT}, the authors ask whether using non-combinatorial techniques instead may lead to stronger lower bounds.

One candidate is the \emph{hyperplane separation bound} proposed by Fiorini \cite{Fio} 
and applied by Rothvo{\ss} \cite{Rot} in his proof of the exponential lower bound for the matching polytope.
It is a lower bound on the extension complexity of a polytope $P$ 
that, for every inequality in a given linear description of $P$, not only depends on whether a vertex is incident with the corresponding face of $P$ but also takes into account the slack of the vertex in the inequality.
This information defines a \emph{slack matrix} of $P$.
We show that, for the slack matrix obtained from Edmonds' \cite{Edm} canonical description of $\Pst(K_n)$,
the hyperplane separation technique fails to produce a lower bound stronger than $\Omega(n^2)$. In this sense, the trivial dimension bound is already at least as strong. 
Our proof in \cref{sec:sptree} relies on a dual interpretation of the method, which will be explained in \cref{sec:prelim}.

At the same time, we stress that our result does not rule out the possibility of obtaining meaningful bounds 
for different slack matrices of $\Pst(K_n)$.
For instance, one may rescale the inequalities describing $\Pst(K_n)$ or add redundant linear inequalities to the description.
\Cref{sec:scale} studies the effect of these operations on the hyperplane separation bound.
In particular, the hyperplane separation bound is not invariant under scaling the rows and columns of a given slack matrix. 
This is a property that is shared with the norm-based lower bounds of similar flavour introduced by Fawzi and Parrilo \cite{FPa,FPb}. Which scalings of rows and columns produce the strongest bounds is left as an open question in \cite{FPa}.
We address this issue in \cref{sec:scale} and provide a partial answer: If one rescales the rows in such a way that the maximum entry in every row equals one, and proceeds analogously with the columns, the hyperplane separation bound will not decrease.

The limitations of the hyperplane separation method can be observed in another family of well-understood polytopes as well. 
Given a graph $G=(V,E)$ on $V=[n]$, a \emph{graphic zonotope} of $G$ is the Minkowski sum of $|E|$ line segments in the 
directions $\{u^j - u^i \}_{ij \in E}$ (see~\cite{PRW}), where $u^i$ denotes the $i$th canonical unit vector in $\R^n$.
Every graphic zonotope of $G$ is the affine linear image of the hypercube $[0,1]^{|E|}$ and, hence, its extension
complexity is at most $2|E| \le n(n-1)$.
In fact, no smaller extended formulation is known to date, not even for \emph{completion time polytopes},
a well-known subclass of graphic zonotopes of $K_n$. Their facets have been described by Wolsey \cite{Wol} 
(who also first observed the fact that they are zonotopes, see the remark in \cite{KP}) and Queyranne \cite{Que}.
Completion time polytopes will be discussed in more detail in \cref{sec:zonotopes}.

Arguably the simplest of all completion time polytopes is the $n$th \emph{permutahedron}, which is defined as 
$
  \conv\{ (\pi(1),\dots,\pi(n)) \colon \pi \in \Symm \},
$
where $\Symm$ denotes the symmetric group on $[n]$.
For this polytope, Goemans \cite{Goe} gives an asymptotically minimal extended formulation of size $\Theta(n \log n)$.
The lower bound in \cite{Goe} is established via a purely combinatorial argument.
Since any two graphic zonotopes of $K_n$ are combinatorially equivalent (see \cref{sec:zonotopes}), $\Omega(n \log n)$ is therefore best possible for any combinatorial lower bound on the extension complexity of graphic zonotopes of $K_n$.

In \cref{sec:zonotopes}, we first give a description of the facets of graphic zonotopes. This generalizes the canonical linear description of completion time polytopes in \cite{Wol,Que}.
For this description and the resulting slack matrix, we then show that the hyperplane separation bound is at most a constant.

%-------------------------------------------------------------------------------------

\section{Slack matrices and the hyperplane separation bound} \label{sec:prelim}

Given a nonnegative matrix $S \in \R^{m \times n}_{\ge 0}$, the \emph{nonnegative rank} of $S$, denoted by $\rkn(S)$, 
is defined as the minimum $r \in \N$ such that $S = UV$ for two nonnegative matrices $U \in \R^{m \times r}_{\ge 0}, V \in \R^{r \times n}_{\ge 0}$.
Equivalently, it is the minimum $r \in \N$ such that $S$ can be written as the sum of $r$ nonnegative matrices of rank one \cite{CR}.

Consider a polytope $P = \conv(X) = \{ x \in \R^n \colon Ax \le b, x \in \aff(P) \}$ for some finite set 
$X=\{x^1,\dots,x^v\} \subset \R^n$ and $A \in \R^{f \times n}, b \in \R^f$ such that every inequality in $Ax \le b$ defines a nonempty face of $P$.
The $f \times v$ matrix whose $j$th column equals $b-Ax^j$ is a \emph{slack matrix} of $P$.
If $X$ is the set of vertices of $P$, we refer to the corresponding slack matrix as
the slack matrix of $P$ \emph{with respect to} the linear description above.
In particular, any slack matrix of a polytope is a nonnegative matrix whose nonnegative rank satisfies the following property
due to Yannakakis \cite{Yan}.
\begin{proposition} \label{prop:Yannakakis}
Let $S$ be a slack matrix of a polytope $P$. Then $\xc(P) = \rkn(S)$.
\end{proposition}

This result is the key to many techniques for bounding the extension complexity of $P$.
This paper is concerned with one such technique.
For two matrices $A=(a_{ij})$, $B=(b_{ij}) \in \R^{m \times n}$, we let $\norm{A} := \max_{i,j} |a_{ij}|$
and denote by 
$\langle A,B \rangle := \sum_{i=1}^m \sum_{j=1}^n a_{ij} b_{ij}$
the \emph{Frobenius inner product} of $A$ and $B$.
We will use the same notation $\langle a,b \rangle$ for the inner product of two vectors $a,b \in \R^n$.

\begin{proposition}[Hyperplane separation bound 
  \cite{Rot}%
] \label{prop:HSB}
Let $S \in \R_{\ge 0}^{m \times n}$ not identically zero, 
and let $\mathcal{R}_{m,n}$ denote the set of rank-one matrices in $\{0,1\}^{m \times n}$. We further let
\begin{equation} \label{eq:HSB_def}
  \hsb(S) := \sup \left\{ 
    \frac{\left\langle S,X \right\rangle}{\norm{S} \rho(X)} \colon X \in \R^{m \times n}
  \right\},
\end{equation}
where $\rho(X) := \max \left\{ \left\langle X,R \right\rangle \colon R \in \mathcal{R}_{m,n}\right\}$ for every $X \in \R^{m \times n}$.
Then $ \rkn(S) \ge \hsb(S). $
\end{proposition}

Normalizing $X$ such that $\rho(X) = 1$ in the definition of $\hsb(S)$, 
we may rewrite \cref{eq:HSB_def} as follows: 
\begin{align} \label{eq:HSB_primal}
  \norm{S} \hsb(S) 
  &= \sup \left\{ 
    \left\langle S,X \right\rangle 
    \colon X \in \R^{m \times n},
    \rho(X)=1
  \right\}
  \nonumber \\
  &= \sup \left\{ 
    \left\langle S,X \right\rangle 
    \colon X \in \R^{m \times n},
    \rho(X) \le 1
  \right\} 
  \nonumber \\
  &= \max \left\{ 
    \left\langle S,X \right\rangle 
    \colon X \in \R^{m \times n},
    \left\langle X,R \right\rangle \le 1 \;\forall R \in \mathcal{R}_{m,n}
  \right\}.
\end{align}
In the last step, we used the fact that the supremum of $\langle S,\cdot \rangle$ is finite:
Any $X \in \R^{m \times n}$ with $\rho(X) \le 1$ satisfies $\langle X,R \rangle \le 1$ for all $R$ with singleton support,
that is, every entry of $X$ is at most one.
As $S$ is nonnegative, the sum of its entries is an upper bound on $\langle S,X \rangle$.

Note that \cref{eq:HSB_primal} is a linear program (LP).
From strong LP duality, we obtain the following dual characterization of the hyperplane separation bound,
which already appears in \cite{Sei} and, in a more general context, in \cite{FPb}.
\begin{proposition} \label{prop:HSB_dual}
With $S$ and $\mathcal{R}_{m,n} =: \mathcal{R}$ defined as in \cref{prop:HSB}, we have
\begin{equation} \label{eq:HSB_dual}
  \hsb(S) = 
  \min \left\{
    \norm{S}^{-1}
    \sum_{R \in \mathcal{R}} y_R \colon\; 
    y \in \R^\mathcal{R}_{\ge 0}, \; \sum_{R \in \mathcal{R}} y_R R = S
  \right\}.
\end{equation}
\end{proposition}

The feasible region of the LP in \cref{eq:HSB_dual} corresponds to a particular type of nonnegative factorization of $S$, namely the decomposition of $S$ into the weighted sum of 0/1 matrices of rank one. 
Thus, if $S$ has a nonnegative factorization of rank $r$ whose factors are 0/1 matrices, then $\hsb(S) \le r / \norm{S}$ by \cref{prop:HSB_dual}.
This observation will be the key ingredient of our proofs in \cref{sec:applications}.

The hyperplane separation bound $\hsb(S)$ is invariant under multiplying $S$ by positive scalars, under transposition, and under permutations of rows and columns of $S$, respectively.
It further satisfies
the following two useful properties on submatrices, both of which are immediate consequences of \cref{prop:HSB_dual}.
\begin{lemma} \label{cor:HSB_submatrix}
Let $S = (A \;\; B)$
for nonnegative matrices $A$ and $B$. Then
\begin{myenumerate}
  \item $\norm{S} \hsb(S) \ge \norm{A} \hsb(A)$,
    \label{enum:HSB_submatrix_ii}
  \item $\hsb(S) \le \hsb(A)+\hsb(B)$.
    If $A=B$, then $\hsb(S) = \hsb(A)$.
    \label{enum:HSB_submatrix_i}
\end{myenumerate}
\end{lemma}

Recall that any two slack matrices of a given polytope have identical nonnegative rank. (This is a consequence of
\cref{prop:Yannakakis}.) In this sense, the nonnegative rank is well-defined for polytopes.
The situation for the hyperplane separation bound, however, is fundamentally different. Before we address this issue more generally in the next section, let us highlight the difference with two examples.

Consider the standard hypercube $C_n = [0,1]^n$ and let
$S_n$ denote its slack matrix w.r.t.\ the (minimal) description $C_n = \{ x \in \R^n \colon 0 \le x_i \le 1,\, i=1,\dots,n\}$.
The inequality $\sum_{i=1}^n x_i \ge 0$ is valid for $C_n$ (defining the vertex $0 \in \R^n$).
Adding this inequality to the minimal description of $C_n$ adds one row to $S_n$, 
which equals the sum of the rows corresponding to the facets defined by $x_i \ge 0$ for $i=1,\dots,n$.
Let $S'_n$ denote the slack matrix with this additional row.
Then we have $\norm{S'_n} = n$ 
and $\norm{S_n} \hsb(S_n) = \norm{S'_n} \hsb(S'_n)$.
Thus, $\hsb(S'_n) = \frac{1}{n} \hsb(S_n)$. 

Not even slack matrices w.r.t.\ \emph{minimal} linear descriptions behave identically under the hyperplane separation bound:
The $n$-simplex spanned by the canonical unit vectors in $\R^n$ and the origin is the set of all $x \in \R^n$ satisfying
$x_1 + \dots + x_n \le 1, x_i \ge 0$ for $i=1,\dots,n-1$, and $\lambda x_n \ge 0$ for any $\lambda\ge 1$.
Every inequality defines a facet of the simplex.
Modulo permutations of rows and columns, the associated slack matrix $S_{n,\lambda}$ is obtained from 
the $(n+1) \times (n+1)$ identity by multiplying the first row by $\lambda$.
One can show that $\hsb(S_{n,\lambda}) = \frac{n}{\lambda} +1$ while $\rkn(S_{n,\lambda}) = n+1$.

%-------------------------------------------------------------------------------------

\section{Diagonal scalings and redundancy} \label{sec:scale}

Let us first study the effect of scaling rows or columns of a nonnegative matrix $S$.
A \emph{positive diagonal scaling} of $S$ is a matrix $S'$ which can be written as $S' = D_1 S D_2$ 
where $D_1$ and $D_2$ are positive diagonal matrices, i.e., diagonal matrices with positive diagonal elements. 
Note that $\rkn(S') = \rkn(S)$ \cite{CR}, while in the examples in \cref{sec:prelim} we have seen that the hyperplane separation bound may indeed change (cf.~\cite{FPa,FPb}). 
The following lemma is our main ingredient in this section.

\begin{lemma} \label{lem:HSB_scale}
Let $S \in \R^{m \times n}_{\ge 0}$ not identically zero, 
and let $D \in \R^{m \times m}_{\ge 0}$ be a positive diagonal matrix. Then
\[
  \hsb(D S) \le \hsb(S) \frac{\norm{D} \norm{S}}{\norm{D S}}.
\]
\end{lemma}
\begin{proof}
Let $X \in \R^{m \times n}$ be a feasible solution of the LP in \cref{eq:HSB_primal} that is optimal for $S' := D S$.
We will denote the $i$th row of $S$ and $X$ by $s^i$ and $x^i$, respectively, and the $i$th diagonal element of $D$ by $d_i>0$.
First, observe that $\langle s^i,x^i \rangle \ge 0$ for all $i=1,\dots,m$:
Indeed, if $\langle s^i,x^i \rangle < 0$ for some $i$, let $S'_{-i}$ and $X_{-i}$ be the matrices obtained from $S'$ and $X$
by deleting the $i$th row. Then
\[
  \norm{S'_{-i}} \hsb(S'_{-i}) \ge \langle S'_{-i}, X_{-i} \rangle 
  = \langle S', X \rangle - d_i \langle s^i,x^i \rangle 
  > \langle S', X \rangle = \norm{S'} \hsb(S'),
\]
contradicting \cref{cor:HSB_submatrix}\cref{enum:HSB_submatrix_ii}.
We conclude that
\[
  \langle S', X \rangle = \sum_{i=1}^m d_i \langle s^i,x^i \rangle
  \le \sum_{i=1}^m \norm{D} \langle s^i,x^i \rangle = \norm{D} \langle S,X \rangle
  \le \norm{D} \norm{S} \hsb(S).
  \qedhere
\]
\end{proof}

Since the hyperplane separation bound is invariant under transposition, \cref{lem:HSB_scale} immediately generalizes
to positive diagonal scalings.
\begin{theorem} \label{cor:HSB_scale}
Let $S \in \R^{m \times n}_{\ge 0}$ not identically zero, 
and let $D_1 \in \R^{m \times m}_{\ge 0}, D_2 \in \R^{n \times n}_{\ge 0}$ be positive diagonal matrices. Then %we have
\[
  \frac{\norm{S}}{\norm{D_1^{-1}} \norm{D_1 S D_2} \norm{D_2^{-1}}} \hsb(S) 
  \le \hsb(D_1 S D_2) \le \hsb(S) \frac{\norm{D_1} \norm{S} \norm{D_2}}{\norm{D_1 S D_2}}.
\]
\end{theorem}
\begin{proof}
Follows from \cref{lem:HSB_scale}, using the fact that $S = D_1^{-1} (D_1 S D_2) D_2^{-1}$.
\end{proof}

Consider a positive diagonal scaling $S'$ of $S$ whose nonzero rows and columns are normalized w.r.t.\ $\norm{\cdot}$. We say that $S'$ is both \emph{row-} and \emph{column-normalized}.
If $S' = D_1 S D_2$ with $D_1$ and $D_2$ chosen in such a way that $D_1 S$ is row-normalized or $S D_2$ is
column-normalized, then $\norm{D_1^{-1}} \norm{D_2^{-1}} = \norm{S}$ and \cref{cor:HSB_scale} implies that $\hsb(S') \ge \hsb(S)$.
In general, though, not every row- and column-normalized diagonal scaling $S'$ results from a pair of diagonal matrices
that satisfies this additional requirement.

Rescaling rows and columns is not the only operation that has an effect on the hyperplane separation bound of a nonnegative matrix $S$. 
Suppose that we add to $S$ a row (column) which is a nonnegative linear combination of rows (columns) of $S$ 
and is therefore redundant (cf.~the example in \cref{sec:prelim}). This operation, too, leaves the nonnegative rank of $S$ unchanged \cite{CR}.
The next lemma bounds the gain on $\hsb(S)$.

\begin{lemma} \label{lem:HSB_addrow}
Let $S \in \R^{m \times n}_{\ge 0}$ and $S' := \begin{pmatrix} S \\ wS \end{pmatrix}$
for a row vector $w \in \R^m_{\ge 0}$ where $\norm{wS} \le \norm{S}$.
Then we have 
\[
  \hsb(S) \le \hsb(S') \le \hsb(S) \, \max\{1, \norm[1]{w}\}.
\]
\end{lemma}
\begin{proof}
The first inequality follows from \cref{cor:HSB_submatrix}\cref{enum:HSB_submatrix_ii}.
In order to show the second inequality, suppose first that $\norm[1]{w} \le 1$ and
let $X \in \R^{m \times n}$, $x \in \R^m$ such that 
$\left(\begin{smallmatrix} X \\ x \end{smallmatrix}\right) \in \R^{(m+1) \times n}$ is an optimal solution of the LP in \cref{eq:HSB_primal} for $S'$.
Adding $w_i x$ to the $i$th row of $X$ for every $i=1,\dots,m$, we obtain a matrix $X' \in \R^{m \times n}$ satisfying
$\langle S,X' \rangle = \langle S,X \rangle + \langle wS, x \rangle$.
It remains to show that $X'$ is a feasible solution of the LP in \cref{eq:HSB_primal} for $S$.
To this end, let $R \in \mathcal{R}_{m,n}$, and let $I \subseteq [m]$ denote the set of indices of rows that $R$ is supported in.
Note that all rows in $I$ are equal to some $r \in \{0,1\}^n$. Hence,
\[
  \langle X', R \rangle = \langle X,R \rangle + \langle x,r \rangle \sum_{i \in I} w_i .
\]
If $\langle x,r \rangle \le 0$, then $\langle X',R \rangle \le \langle X,R \rangle \le 1$ because $X$ is feasible.
Otherwise, $\langle X',R \rangle \le \langle X,R \rangle + \langle x,r \rangle \le 1$
since $\sum_{i \in I} w_i \le \norm[1]{w} \le 1$ and 
$\left( \begin{smallmatrix} R \\ r \end{smallmatrix} \right) \in \mathcal{R}_{m+1,n}$.

Now suppose that $\norm[1]{w}>1$ and let
\[
  S'' := \begin{pmatrix} S \\ \norm[1]{w}^{-1} w S \end{pmatrix}.
\]
Then $\hsb(S'') = \hsb(S)$ and $S'$ is a diagonal scaling of $S''$ where the maximum diagonal element equals 
$\max\{1,\norm[1]{w}\} = \norm[1]{w}$. The statement follows from \cref{lem:HSB_scale}.
\end{proof}

There is no loss of generality in requiring that $\norm{wS} \le \norm{S}$ above:
If $\norm{wS} > \norm{S}$, one may ``normalize'' the redundant row $wS$ by replacing $w$ with $\norm{S} \norm{wS}^{-1} w$. \Cref{cor:HSB_scale} guarantees that this will not decrease $\hsb(S')$.
Further note that \cref{lem:HSB_addrow} applies to any $w$ with $\norm[1]{w} \le 1$
since $\norm{wS} \le \norm[1]{w}\norm{S}$. In this case, we obtain $\hsb(S')=\hsb(S)$.
Finally, the statement of \cref{lem:HSB_addrow} easily extends to the case of adding multiple rows and, by transposition, columns.

%-------------------------------------------------------------------------------------

\section{Limitations} \label{sec:applications}
\subsection{The spanning tree polytope} \label{sec:sptree}

Let $G=(V,E)$ be a connected graph.
The spanning tree polytope of $G$ given in \cref{eq:st_polytope}
is completely described by the following system due to Edmonds \cite{Edm}:
\begin{equation} \label{eq:sptree}
  \Pst(G) =
  \left\{ 
    x \in \R^E_{\ge 0} \colon
    x(E) = |V|-1, \;
    x(E(U)) \le |U|-1 \quad \forall\, \emptyset \ne U \subset V
  \right\},
\end{equation}
where $E(U)$ is the set of all edges with both endpoints in $U$.
\begin{theorem} \label{thm:HSB_sptree}
Let $G=(V,E)$ be a connected graph 
and let $S_G$ denote the slack matrix of $\Pst(G)$ w.r.t.\ the description \cref{eq:sptree}. Then
$
  \hsb(S_G) = O(|E|).
$
\end{theorem}

\begin{proof}
Since there are $|E|$ many nonnegativity constraints in \cref{eq:sptree}, it suffices to consider 
the row submatrix of $S_G$ restricted to the set inequalities in \cref{eq:sptree} only, which will be denoted by $S_G$ again. 
The bound for the entire slack matrix then follows from \cref{cor:HSB_submatrix}\cref{enum:HSB_submatrix_i}.

We shall index the rows of $S_G$ by the nonempty subsets of $V$ and the columns by the spanning trees in $G$.
The entry in row $U \subset V$ and column $T$ equals $c(U,T)-1$, where $c(U,T)$ denotes the number of connected components of the subgraph $(U,T \cap E(U))$.
First, observe that
\begin{equation} \label{eq:infnorm_sptree}
  \norm{S_G} \ge \tfrac{1}{2} |V|-1.
\end{equation}
For, if $T$ is a spanning tree in $G$ and $U \subset V$ a stable set in $T$,
then $c(U,T) = |U|$.
Because $T$ is a bipartite graph, both vertex classes in a bipartition are stable sets in $T$. At least one of them is of size $|V|/2$.

Based on Martin's extended formulation \cite{Mar}, Conforti et al.\ \cite{CCZ} show that $S_G$ admits a nonnegative factorization 
of rank $O(|V| \, |E|)$ where both factors are 0/1 matrices. From \cref{eq:infnorm_sptree} and \cref{prop:HSB_dual}, 
we conclude that $\hsb(S_G) = O(|E|)$.
\end{proof}

Note that normalizing the rows of $S_G$ (by rescaling the inequalities in \cref{eq:sptree})
produces a matrix $S_G'$ with $\hsb(S_G) \le \hsb(S_G') \le \norm{S_G} \hsb(S_G) = O(|V| \, |E|)$ by \cref{cor:HSB_scale}.
(The columns of $S_G'$ are normalized as well.) 

%-------------------------------------------------------------------------------------

\subsection{Graphic zonotopes} \label{sec:zonotopes}

A line segment in $\R^n$ is a set $[x,y] := \conv(\{x,y\})$ for some $x,y \in \R^n$. 
Recall from \cref{sec:intro} that a graphic zonotope of a graph $G=(V,E)$ with $V=[n]$ 
is the Minkowski sum of a finite number of line segments, each of which is parallel to $u^j - u^i$ for some $ij \in E$.
Let $A = (a_{ij}) \in \R^{n \times n}_{\ge 0}$ be a symmetric nonnegative matrix. 
We associate with $A$ a zonotope $Z(A) \subset \R^n$ as follows:
\[
  Z(A) := \sum_{1 \le j \le n} a_{jj} u^j + \sum_{1 \le i<j \le n} a_{ij} [u^i, u^j].
\]
Up to translations, the graphic zonotopes of graphs on $n$ vertices are exactly those of the above form for some 
symmetric and nonnegative matrix $A$ (where $a_{ij}>0$ if and only if $ij \in E$).

We will now derive a description of the facets of $Z(A)$, generalizing remarks in \cite[Example~7.15]{Zie} and \cite{Fuj}.
To this end, define the set function $g_A \colon 2^{[n]} \to \R$ by 
\[ 
  [n] \supseteq S \mapsto g_A(S) := \sum\limits_{\substack{i,j \in S \colon \\ i \le j}} a_{ij}.
\]
Note that $g_A$ is supermodular,
and it is strictly supermodular if and only if $A$ is positive.
Using standard arguments (see \cite{Edm70}), one can show that $Z(A)$ is the \emph{supermodular base polytope} 
(see, e.g., \cite{Fuj}) of $g_A$,
\begin{equation} \label{eq:supmod_polytope}
  Z(A) = \left\{ x \in \R^n \colon x([n]) = g_A([n]), \, x(S) \ge g_A(S) \;\forall S \subseteq [n] \right\}.
\end{equation}
Its vertices are in correspondence with the permutations in $\Symm$ via the map
\begin{equation} \label{eq:bij}
\begin{split}
  \Symm \ni \pi \quad\longmapsto\quad 
  x^\pi \in \R^n \,;  \qquad %\\
  x_j^\pi 
    = \sum_{\substack{ i \in [n] \colon \\ \pi(i) \le \pi(j) }} a_{ij} 
  \; , \quad\, j =1,\dots,n.
\end{split}
\end{equation}

We define a matrix $M_A$ with one row for every nontrivial subset of $[n]$ and one column for every permutation in $\Symm$ as follows:
If $x^\pi$ denotes the vertex of $Z(A)$ induced by $\pi \in \Symm$ via \cref{eq:bij}, the entry of $M_A$ 
in row $S \subsetneq [n], S \ne \emptyset$, and column $\pi$ equals
\begin{equation} \label{eq:zonotope_slack}
  x^\pi(S) - g_A(S)
  = \sum_{\substack{i \in [n], j \in S \colon \\ \pi(i) \le \pi(j)}} a_{ij} 
    - \sum_{\substack{i,j \in S \colon \\ \pi(i) \le \pi(j)}} a_{ij}
  = \sum_{\substack{i \notin S, j \in S \colon \\ \pi(i) \le \pi(j)}} a_{ij},
\end{equation}
using symmetry of $A$ in the first equation.
Thus, $M_A$ is precisely the slack matrix of $Z(A)$ w.r.t.\ its linear description \cref{eq:supmod_polytope}, possibly with repeated columns.

Before we state the main result of this section, observe that the support of $M_A$ is independent of the actual entries of $A$ if $A$ is strictly positive. In this case, \cref{eq:bij} defines a bijection and
all inequalities in \cref{eq:supmod_polytope} for $\emptyset \ne S \subsetneq [n]$ define facets of $Z(A)$ 
since $g_A$ is strictly supermodular \cite{Que}.
Hence, $Z(A)$ and $Z(A')$ are combinatorially equivalent for any two symmetric $A,A' \in \R^{n \times n}_{>0}$.

\begin{theorem} \label{thm:zonotope_hsb}
Let $A \in \R^{n \times n}_{\ge 0}$ be symmetric, 
and let $M_A$ be the slack matrix of $Z(A)$ w.r.t.\ \cref{eq:supmod_polytope}. Then
$
  \hsb(M_A) \le 4.
$
\end{theorem}

\begin{proof}
For every pair $i,j \in [n], i \ne j$, let
\[
  R(i,j) := \{ S \subset [n] \colon i \notin S, j \in S \} \times \{ \pi \in \Symm \colon \pi(i) \le \pi(j) \}
\]
and let $\widehat{R}(i,j)$ denote the unique 0/1 matrix indexed like $M_A$ whose support equals $R(i,j)$.
Note that $\widehat{R}(i,j)$ has rank one and, by \cref{eq:zonotope_slack}, 
\begin{equation*}
  \sum_{i \ne j} a_{ij} \widehat{R}({i,j}) = M_A.
\end{equation*}
Since the expression in \cref{eq:zonotope_slack} is less than or equal to $\sum_{i \notin S, j \in S} a_{ij}$
with equality if $\pi([n] \setminus S) = [n-|S|]$,
we have that
\[
  \norm{M_A} = 
    \max\limits_{S \subseteq [n]} \sum_{i \notin S, j \in S} a_{ij}.
\]
This is, in fact, the maximum weight of a cut in the graph underlying the graphic zonotope $Z(A)$
where the (nonnegative) edge weights are given by $A$.
Since there is always a cut whose weight is at least half the total weight of the edges, we conclude that
$\norm{M_A} \ge \frac{1}{2} \sum_{i < j} a_{ij} = \frac{1}{4} \sum_{i \ne j} a_{ij}$.
The theorem follows from an application of \cref{prop:HSB_dual}.
\end{proof}

In the light of \cref{sec:scale}, it seems possible that a nontrivial bound may be obtained by rescaling the rows and columns 
of the slack matrix $M_A$. In fact, as a byproduct of the proof above, we have that the maximum in row $S \subseteq [n]$ of 
$M_A$ equals the weight of the cut induced by $S$ in the graph underlying $Z(A)$.
Therefore, the gain on $\hsb(M_A)$ that one can expect from normalizing the rows of $M_A$ according to \cref{cor:HSB_scale}
(which actually produces a matrix $M_A'$ that is column-normalized as well) is at most the ratio of the maximum and the minimum weight of a cut. 
While this ratio can grow arbitrarily large, it depends on $A$.
For instance, if $A$ is the $n \times n$ all-one matrix, normalizing $M_A$ does not help much: 
Since every cut in $K_n$ has at least $n-1$ and at most $\lfloor n/2 \rfloor \lceil n/2 \rceil$ edges,
we obtain $\hsb(M_A') = O(n)$.

This special case, however, falls into a well-known subclass of graphic zonotopes.
Consider $n$ jobs with processing times $p=(p_1,\dots,p_n) \in \R^n_{>0}$ to be scheduled on a single machine.
Every permutation $\pi~\in~\Symm$ defines a feasible schedule without idle time
where job $j$ is completed at time $C_j^\pi := \sum_{i \colon \pi(i) \le \pi(j)} p_i$ for $j = 1,\dots,n$.
The \emph{completion time polytope} $\Pct(p)$ is defined as
\[
  \Pct(p) := \conv \left\{
    (C_1^\pi,\dots,C_n^\pi) \in \R^n \colon \pi \in \Symm
  \right\}.
\]
Now let $A \in \R^{n \times n}_{>0}$ be a positive rank-one matrix. It can be shown that $A = pp^T$ for some $p \in \R^n_{>0}$.
Then $Z(A)$ is the image of $\Pct(p)$ under the linear transformation $(x_1,\dots,x_n) \mapsto (p_1 x_1, \dots, p_n x_n)$.
Up to this transformation, the inequalities in \cref{eq:supmod_polytope} coincide with the canonical linear description 
of $\Pct(p)$ due to Wolsey \cite{Wol} and Queyranne \cite{Que}.
In case that $p_j=1$ for all $j=1,\dots,n$, $\Pct(p)$ is the $n$th permutahedron and is equal to $Z(A)$ with the $n \times n$ all-one matrix for $A$.

%-------------------------------------------------------------------------------------

\section{Concluding remarks}

For both families of polytopes studied in this paper and their canonical slack matrices, 
we have shown that the hyperplane separation technique is unable to improve on the currently best known lower bounds on
their extension complexity.
In contrast to the nonnegative rank, the hyperplane separation bound depends on the choice of slack matrix.
By making a more careful choice, it is conceivable that the technique
does indeed yield more meaningful bounds than the ones in \cref{sec:applications}.

In particular, if one first normalizes the rows of a nonnegative matrix and then normalizes
the columns of the resulting matrix (or vice versa) in the sense of \cref{sec:scale}, this
will only strengthen the hyperplane separation bound while preserving the nonnegative rank.
However, it is not clear which positive diagonal scalings yield the strongest bounds among those that are both row- and column-normalized.
For the polytopes considered in \cref{sec:applications}, normalizing the rows of their canonical slack matrices 
is already sufficient to obtain a row- and column-normalized matrix. How much can one gain by this?
Although there is hope that normalizing may indeed overcome the negative results of \cref{thm:HSB_sptree,thm:zonotope_hsb}, we leave this as an open question.

Another potential way of strengthening the hyperplane separation bound is redundancy achieved by adding nonnegative linear combinations of rows or columns. How does this compare to the gain that is achieved by the best diagonal scalings?
At least adding a row (column) which is a convex combination of rows (columns) has no effect on the hyperplane separation bound.

%-------------------------------------------------------------------------------------

\medskip \bigskip
\noindent\textbf{Acknowledgements.}
The author is grateful to Andreas S. Schulz and Stefan Weltge for helpful discussions and comments.
He would also like to thank an anonymous referee whose comments on an earlier version greatly improved the paper.

%-------------------------------------------------------------------------------------

\end{document}